\documentclass[12pt]{amsart}
\usepackage{amsmath}
\usepackage{amssymb}
\usepackage{amscd}

\newtheorem{Theorem}{Theorem}[section]
\newtheorem{Lemma}[Theorem]{Lemma}

\newtheorem{Proposition}[Theorem]{Proposition}
\newtheorem{Remark}[Theorem]{Remark}

\newtheorem{Example}[Theorem]{Example}

\def\depth{\operatorname{depth}}

\def\reg{\operatorname{reg}}

\def\mm{{\mathfrak m}}

\def\ZZ{{\mathbb Z}}
\def\NN{{\mathbb N}}

\begin{document}
\title{Castelnuovo-Mumford regularity of \\
 associated graded modules  in dimension one}
\thanks{The author was partially supported by NAFOSTED
(Vietnam).}

\maketitle

\begin{center}
LE XUAN DUNG\\
Department of  Natural Sciences,  Hong Duc University\\
307 Le Lai, Thanh Hoa, Vietnam\\
E-mail: lxdung27@gmail.com\\ [15pt]
%and\\[15pt]
%LE TUAN HOA \\
 %Institute of Mathematics\\ 18 Hoang Quoc Viet, 10307 Ha Noi, Vietnam\\
 %E-mail: lthoa@@math.ac.vn
  \end{center}

\begin{abstract} An  upper bound for  the Castelnuovo-Mumford regularity of the  associated graded
module of an one-dimension module is given  in term of its Hilbert coefficients. It is also investigated
 when the bound is attained.
\end{abstract}

\date{}
\section{Introduction} 
Let $A$ be  a local ring with the maximal ideal $\mm$ and  $M$ be a finitely generated  $A$-module of dimension $r$.
Denote by $G_I(M) = \oplus_{n\geq 0}I^nM/I^{n+1}M $  the associated graded module of $M$
 with respect to  an $\mm$-primary ideal $I$. Because of the importance of the Castelnuovo-Mumford
  regularity $\reg(G_I(M))$ of $G_I(M)$, it is of interest to bound it in terms of other invariants of $M$.
  Rossi, Trung and  Valla \cite{RTV} and Linh \cite{L}   gave bounds on $\reg(G_I(M))$  in terms of
   the so-called extended degree of $M$ with respect to $I$.

On the other hand, Trivedi \cite{Tr1} and Brodmann and Sharp \cite{BS} gave bounds for so-called  Castelnuovo-Mumford
  regularity at and above level 1  in terms of  the  Hilbert coefficients  $e_0(I,M), e_1(I,M), ...,e_{r-1}(I,M)$.
Using also $e_r(I,M)$, Dung and Hoa  \cite{HD} can bound $\reg(G_I(M))$. Bounds in \cite{BS,HD,Tr1} are very big:
they are exponential functions of  $r!$. In the
general case, it follows from \cite[Lemma 11 and Proposition 12]{HH} that in the worst case $\reg(G_I(M))$  must be a double
exponential function of $r$. However one may hope that in small dimensions,  one can give small and sharp bounds.

The aim of this paper is to give a sharp bound in dimension one case.

\medskip

\noindent {\bf Theorem }  {\it Let  $M$ be an one-dimensional module. Let $b$ be the maximal integer such that $IM \subseteq \mm^bM.$ Then
$$\reg(G_I(M)) \le {e_0 -  b+2\choose 2}  - e_1 -1.$$}

\medskip
\noindent  We also give characterizations for the case that the bound is attained (see Theorem \ref{DLB3} and Theorem \ref{DLB4}).

The paper is divided into two sections. In Section 1 we start with a few preliminary results on bounding the
Castelnuovo-Mumford   regularity of  graded modules of dimension at most one. Then we apply these results to derive a bound on $\reg(G_I(M))$
 provided $M$ is a Cohen-Macaulay module of dimension one (Proposition \ref{BDA7}). Using a relation
 between Hilbert coefficients $e_0(I,M)$ and $e_1(I,M)$ given in Rossi-Valla \cite{RV} we can deduce and prove the above bound.

 In Section 2  we give  characterizations for  the case that  the bound is  attained. The characterization is given in terms of the Hilbert-Poincare series (Theorem \ref{DLB4}). In the case of Cohen-Macaulay modules there is also a characterization   in terms of the  Hilbert coefficients (Theorem \ref{DLB3}).

\section{An upper bound} \label{A}

Let $R= \oplus_{n\ge 0}R_n$ be a Noetherian standard graded ring over a local Artinian ring $(R_0,\mm_0)$. Since tensoring with $(R_0/\mm_0)(x)$ doesn't  change invariants considered in this paper, without loss of generality we may always assume that $R_0/\mm_0$ is an infinite field. Let $E$ be a finitely generated graded module of dimension $r$.

First let us recall some notation. For $0\le i \le r$, put
 $$ a_i(E) =
\sup \{n|\ H_{R_+}^i(E)_n \ne 0 \} ,$$
where $R_+ = \oplus _{n>  0} R_n$. The {\it Castelnuovo-Mumford regularity} of $E$  is defined by
$$\reg(E) = \max \{ a_i(E) + i \mid  0\le  i \leq  r  \},$$
and the {\it Castelnuovo-Mumford regularity of $E$  at and above level $1$},   is defined by
$$\reg^1(E) = \max \{ a_i(E) + i \mid  1\leq i \leq  r \}.$$

We denote the Hilbert function $\ell_{R_0}(E_t)$ and the Hilbert polynomial of $E$ by $h_E(t)$ and $p_E(t)$, respectively. Writing $p_E(t)$ in the form:
$$p_E (t) = \sum_{i=0}^{r-1} (-1)^i e_i(E){t+r-1-i \choose r-1-i},$$
we call  the numbers $e_i(E)$  {\it Hilbert coefficients} of $E$.

We know that $h_E(t) = p_E(t)$ for all $t \gg 0.$ The  {\it postulation number} of a finitely generated graded $R$-module $E$  is defined as the number
$$p(E): = \max\{t \mid h_E(t)\neq p_E(t)\}.$$
This number can be read off from the {\it Hilbert-Poincare  series} of $E$, which is defined as follows:
$$HP_E(z) = \sum_{i\geq 0}h_E(i)z^i.$$
The following result is well-known, see e. g.  \cite[Lemma 4.1.7 and Proposition 4.1.12]{BH}.

\begin{Lemma}\label{BDA0} There is a polynomial $Q_E(z)  \in \ZZ[z]$ such that $Q_E(1) \neq 0$ and
$$HP_E(z) = \frac{Q_E(z)}{(1-z)^r}.$$
 Moreover,  $ p(E) = \deg(Q_E(z)) -r.$
\end{Lemma}

\begin{Remark}\label{BDA1}{\rm
From the Grothendieck-Serre  formula
\begin{eqnarray}\label{a}
h_E(t)  - p_E(t) = \sum_{i=0}^r(-1)^i\ell(H_{R_+}^{i}(E)_t),
\end{eqnarray}we easily get
\begin{itemize}
 \item[(i)] If $\depth(E) = 0$ then $p(E) \leq \reg(E)$;
 \item[(ii)] If $\depth(E) > 0$ then $p(E) \leq \reg(E) -1$.
 \end{itemize}}
\end{Remark}

Let $d(E)$ denote the maximal generating degree of $E$. For short, we also write $d: =d(E)$ and $e: = e_0(E)$. The proof of the following result is similar to that of  \cite[Lemma 4.2]{H}.

 \begin{Lemma} \label{BDA2}Assume that  $\dim(E) = 0$.  Let $q: = \sum_{i\leq d}\ell(E_i)$.   Then
\begin{itemize}
 \item[(i)]  $\ \reg(E) \leq d + e - q$.
 \item[(ii)] The following conditions are equivalent:

                \begin{itemize}
 \item[(a)] $\reg(E)  = d + e - q$;
 \item[(b)]
 $ h_E(t) =
\begin{cases}
1  \ \ \ \ \text{if} \  d + 1 \leq t \leq d + e -q ,\\
0  \ \ \ \ \text{if} \ \ t \geq d + e -q +1.
\end{cases}$
                \end{itemize}
 \end{itemize}
\end{Lemma}
\begin{proof} (i) Let $m = \reg(E)$. Since $\dim(E) = 0$, $ m = \max\{t\mid E_{t} \neq 0\}$. Note that $E_i \neq 0$ for all $d \leq i \leq m$. Hence
\begin{equation}\label{c1.1}
m \leq d + \ell(E_{d+1}\oplus \cdots \oplus E_m) =d  +\ell(E)-q =d + e -q.
\end{equation}
(ii) By (\ref{c1.1}) is it clear that $m =  d + e -q$ if only if $
\ell(E_i) = 1 \text{ for all } d + 1 \leq i \leq m$. Since $
\ell(E_i) = 0 \text{ for all } i > m$, we get the equivalence of  (a) and  (b).
\end{proof}

\begin{Lemma}\label{BDA4} Let $E$ be an one-dimensional  Cohen-Macalay module. Let  $z\in R_1$ be an $E$-regular element and $\rho: = \ell(E_{d(E/zE)})$. Then
\begin{itemize}
 \item[(i)]  $\ \reg(E) \leq d + e - \rho$.
 \item[(ii)] The following conditions are equivalent:

                \begin{itemize}
 \item[(a)] $\reg(E)  =d + e - \rho$;
 \item[(b)]
 $  h_E(t) =
\begin{cases}
t-d + \rho \ \ \ \ \ \ \ \text{if} \ d + 1 \leq t \leq d + e -\rho  -1 ,\\
e \ \ \ \ \ \ \ \ \ \ \ \ \ \ \ \ \text{if} \  t \geq d + e -\rho;
\end{cases}$
 \item[(c)] $p(E) =d+ e - \rho -1.$
                \end{itemize}
 \end{itemize}
\end{Lemma}
\begin{proof} (i) Since $z\in R_1$  is an $E$-regular element, we have $e(E/zE) = e$, $\reg(E/zE) = \reg(E)$ and
 \begin{eqnarray}\label{c}
\sum_{i \leq t}h_{E/zE}(i) = h_E(t).
\end{eqnarray}
In particular  $\sum_{i \leq d(E/zE)}h_{E/zE}(i)  = \rho$.
Note that  $d(E/zE) \leq d(E)$. By Lemma \ref{BDA2},
\begin{equation}\label{c1.2}
\reg(E/zE) \leq d(E/zE)+e(E/zE)-\sum_{i \leq d(E/zE)}h_{E/zE}(i) \leq d + e - \rho.
\end{equation}
 Hence  $\reg(E) \leq d + e - \rho$.

(ii) (a) $\Longrightarrow$ (b):  If $\reg(E) = d + e - \rho$, then from (\ref{c1.2}) it implies that $d(E/zE) = d(E)$ and $\reg(E/zE) = d(E/zE) + e - \rho$. Using Lemma \ref{BDA2} (ii) and (\ref{c}) we get (b).

(b) $\Longrightarrow$ (c) follows from  Lemma \ref{BDA0}.

 (c) $\Longrightarrow$ (a):  Suppose that  $p(E) =d + e - \rho -1$. By Remark \ref{BDA1} (ii),
 we get $ d + e - \rho \leq \reg(E)$. Since $\reg(E) \leq d + e - \rho$, we get $\reg(E)= d + e - \rho.$
\end{proof}

Let $(A,\mm)$ be a Noetherian  local ring with an infinite residue field $K:= A/\mm$ and $M$ a finitely generated $A$-module of dimension $r$.  In this paper, we always assume that $I$ is an $\mm$-primary ideal.  The {\it associated graded module} of $M$
 with respect to  $I$ is defined by
  $$G_I(M) = \bigoplus_{n\geq 0}I^nM/I^{n+1}M. $$
This is a module over the associated graded ring $G =  \bigoplus_{n\geq 0}I^n/I^{n+1}.$ Let $ HP_{I,M}(z) := HP_{G_I(M)}(z)$. We call    $H_{I,M} (n) = \ell(M/I^{n+1}M)$ the Hilbert-Samuel function of $M$ w.r.t. $I$. This function  agrees with a polynomial - called Hilbert-Samuel polynomial and denote by $P_{I,M}(n)$ - for $n \gg 0$. If we  write
$$P_{I,M} (n) = e_0(I,M){n+r \choose r} - e_1(I,M){n+r-1 \choose r-1} + \cdots + (-1)^r e_r(I,M),$$
then the integers  $e_i: = e_i(I,M)$ for all $i = 1,...,r$ are called {\it Hilbert coefficients} of $M$ with respect to $I$ (see \cite[Section 1]{RV}).

Assume that  M is a Cohen-Macaulay module. Kirby and Mehran \cite{KM} were able to show that $e_1\leq {e_0 \choose 2}$.  This result was improved by Rossi-Valla as follows:

\begin{Lemma}\label{BDA5}  \cite[Proposition 2.8]{RV} Let  $M$ be an one-dimensional  Cohen-Macalay module. Let  $b$ be a positive integer such that $IM \subseteq \mm^b M.$ Then
\begin{itemize}
 \item[(i)]   $\ e_1\leq {e_0 - b +1 \choose 2}$.
 \item[(ii)] The following conditions are equivalent:

                \begin{itemize}
 \item[(a)]   $ e_1= {e_0 - b +1 \choose 2}$;
 \item[(b)]  $\ HP_{I,M}(z) = \frac{b+\sum_{i=1}^{e_0-b}z^i}{1-z}.$
                \end{itemize}
 \end{itemize}
\end{Lemma}
Note that (ii) is not formulated in  \cite[Proposition 2.8]{RV}, but it immediately follows from  the proof of that result.

\begin{Remark} \label{CY0}{\rm
(i) If $ e_1= {e_0 - b +1 \choose 2}$, then from (i) of the above lemma it follows that $b = \max\{t\mid IM \subseteq \mm^tM\}$.

(ii)  If $I = \mm$ and $M=A$ then $b = 1$ and Lemma \ref{BDA5} (ii) is \cite[Proposition 2.13]{EV} (see also \cite[Theorem 3.1]{ERV}).}
\end{Remark}

 Set $\overline{G_I(M)}: = G_I(M)/H^0_{G_+}(G_I(M))$.

\begin{Lemma}\label{BDA6} Let  $M$ be an one-dimensional   module. Let $b$ be a positive integer such that $IM \subseteq \mm^bM.$ Then $$h_{\overline{G_I(M)}}(0) \geq b.$$
\end{Lemma}
\begin{proof} Note that
$$[H_{G_+}^0(G_I(M))]_0 = \frac{T_0}{IM},$$
 where $T_0 = \cup_{n>0}(I^{n+1}M:I^n )= I^{t+1}M:I^t \subseteq M$ for some $t \gg 0.$
 Set $\widetilde{M}: = M/T_0$. Suppose  that $$\mm^{b-1}\widetilde{M} = \mm^{b}\widetilde{M}.$$
 By Nakayama's Lemma, this gives  $\mm^{b-1}\widetilde{M} = 0$.  Consequently, $\mm^{b-1}M \subseteq T_0. $
 Then $\mm^{b-1}I^tM \subseteq I^{t+1}M$. Since  $IM \subseteq \mm^bM,$
 $$I^{t+1}M \subseteq\mm^b I^tM \subseteq \mm I^{t+1}M \subseteq I^{t+1}M.$$
 Hence $I^{t+1}M = \mm I^{t+1}M$. By  Nakayama's Lemma, we get $I^{t+1}M = 0$. This implies  $\dim(M) \leq 0$, a  contradiction.
 Therefore $\mm^{b-1}\widetilde{M} \neq \mm^{b}\widetilde{M}$ and we obtain a strict chain of submodules of $\widetilde{M}$:
 $$\widetilde{M} \supsetneqq  \mm \widetilde{M} \supsetneqq \cdots \supsetneqq \mm^b\widetilde{M}.$$
   Since $\overline{G_I(M)}_0 = M/T_0 = \widetilde{M},$ we must have $$h_{\overline{G_I(M)}}(0) = \ell( \widetilde{M}) \geq \ell( \widetilde{M}/\mm^b\widetilde{M}) \geq b.$$
\end{proof}

The following result was  given by L. T. Hoa  \cite[Theorem 5.2]{H} for rings, but its proof also holds for modules.

\begin{Lemma} \label{BDA6}Let  $M$ be a module of positive depth.  Then
$$a_0(G_I(M)) \leq a_1(G_I(M)) - 1.$$
\end{Lemma}

 We are now in the position to  show the main results of this section.
\begin{Proposition} \label{BDA7}
Let  $M$ be an one-dimensional  Cohen-Macalay module. Let $b$ be the maximal integer such that $IM \subseteq \mm^bM.$ Then
$$\reg(G_I(M)) \leq e_0 - b.$$
\end{Proposition}
\begin{proof}
Since $\depth(M) > 0$, by Lemma \ref{BDA6},
$$\reg(G_I(M)) = \reg^1(G_I(M)) = \reg(\overline{G_I(M)}).$$
Since $\dim(M) = 1$,   $\overline{G_I(M)}$ is a Cohen-Macaulay module. Assume that $x^* \in I/I^2$ is an $\overline{G_I(M)}$-regular element.
Since $G_I(M)$ is generated in degree $0$, $d(\overline{G_I(M)}/x^*\overline{G_I(M)}) = 0$.
Applying  Lemma \ref{BDA4} (i) and Lemma \ref{BDA6} we get
\begin{equation}\label{EB9}
\reg(\overline{G_I(M)}) \leq e_0 - h_{\overline{G_I(M)}}(0) \leq e_0 - b.
\end{equation}
\end{proof}

\begin{Remark}\label{CYA1}{\rm
Under the above assumption, Linh \cite[Corollary 4.5 (ii)]{L} showed  that $\reg(G_I(M)) \leq e_0-1.$  Hence,
 if $b>1$  the above result  improves Linh's bound.}
\end{Remark}

The following result was formulated  in the introduction.

 \begin{Theorem} \label{DLA8}  Let  $M$ be an one-dimensional  module. Let $b$ be the maximal integer such that $IM \subseteq \mm^bM.$ Then
$$\reg(G_I(M)) \le {e_0 - b +2\choose 2}  - e_1 -1.$$
\end{Theorem}
\begin{proof} Let $L: = H_\mm^0(M).$ By \cite[Lemma 4.3]{L} (or see \cite[Lemma 1.9]{DH}) we have
$$\reg(G_I(M)) \le  \reg(G_I(\overline{M})) + \ell(L).$$
By \cite[Proposition 2.3]{RV},  it implies that $ \ell(L) =\overline{ e}_1 - e_1$ and  $\overline{e}_0 = e_0$, where $\overline{e}_0: = e_0(I,\overline{M})$ and $\overline{e}_1: = e_1(I,\overline{M}).$
Note that $IM \subseteq \mm^bM$ implies $I\overline{M} \subseteq \mm^b\overline{M}.$
Since $\overline{M}$  is a Cohen-Macaulay module,  Lemma \ref{BDA5} (i)   says that $
 \overline{e}_1\leq {e_0 - b +1 \choose 2}.$
This gives
$$
\ell(L) \leq {e_0 - b +1 \choose 2} - e_1.
$$
 By Proposition  \ref{BDA7},
$\reg(G_I(\overline{M})) \leq e_0 - b$.
Hence,
$$
\reg(G_I(M)) \le e_0 -b + {e_0 - b +1\choose 2} -e_1
= {e_0 - b +2\choose 2}  - e_1 -1.
$$
\end{proof}

\section{Extremal case} \label{B}
In this section, let ($A,\mm$) be a Noetherian local ring with an infinite residue field and $I$ an $\mm$-primary ideal. Let $M$ be a finite generated $A$-module.
The aim of this section is to give  characterizations for the case that the bound  in Theorem \ref{DLA8} is attained.

From the Grothendieck-Serre  formula (\ref{a}) and Lemma \ref{BDA6} we immediately  get
\begin{Lemma}\label{BDB0} Let  $M$ be a module of positive depth. Then
$$p(G_I(M)) \leq \reg(G_I(M)) -1.$$
\end{Lemma}

\begin{Lemma} \label{BDB1}Let  $M$ be a module with $\dim(M) \geq 1$. Let $L: = H^0_\mm(M)$ and  $\overline{M}: = M/L.$ Then the following conditions are equivalent:
                \begin{itemize}
 \item[(i)]  $\reg(G_I(M)) = \reg(G_I(\overline{M})) + \ell(L)$;
 \item[(ii)]  $HP_K(z) = \sum_{ \reg(G_I(\overline{M})) + 1}^{ \reg(G_I(\overline{M})) + \ell(L)}z^i,$ where $K = \oplus_{n \geq 0}\frac{I^{n+1}M+L\cap I^nM}{I^{n+1}M}.$
                \end{itemize}
\end{Lemma}
\begin{proof}      Note that  $\ell(K) = \ell(L)$. If $L = 0$,   there is nothing to prove. Assume that  $L \neq 0$.

(ii) $\Longrightarrow$ (i): Since $\dim(K) = 0$, $\reg(K) = \reg(G_I(\overline{M})) + \ell(L).$ From the short exact sequence
\begin{equation}\label{EB14}
0 \longrightarrow K \longrightarrow G_I(M)\longrightarrow G_I(\overline{M}) \longrightarrow 0,
 \end{equation}
we see that (see e. g., \cite[Corollary 20.19  (d)]{E})
$$\reg(G_I(M))  = \max\{\reg(K),\reg(G_I(\overline{M}))\} =  \reg(G_I(\overline{M})) + \ell(L).$$

(i) $\Longrightarrow$ (ii):  Set $a = \reg(G_I(\overline{M}))$ and $m = \max\{t\mid K_t \neq 0\}$. It was proved in \cite[Lemma 1.9]{DH} that
$K_t \neq 0$ for all $a +1 \leq t \leq m$. By (\ref{EB14}),
\begin{align*}
a + \ell(L) &= \reg(G_I(\overline{M})) + \ell(L) = \reg(G_I(M)) \leq \max\{a,m\} \\
&= a+\max\{0,m-a\} \leq a + \ell(K) = a+\ell(L).
\end{align*}
This implies $\ell(K) = m - a$, and so $\ell(K_t) = 1$ for $a+1 \leq t \leq m$ and $\ell(K_t) = 0$ for all  other values.

\end{proof}

Now we can state and prove the main results of this section. It is  interesting to mention that the equality $e_1 = {e_0 - b +1 \choose 2}$ implies the Cohen-Macaulayness of $G_I(M)$.  This implication was shown in \cite[Proposition 2.13]{EV} for the case $I = \mm$ and $M  = A$.
\begin{Theorem} \label{DLB3}Let  $M$ be an one-dimensional  Cohen-Macalay module. Let $b$ be a positive integer such that $IM \subseteq \mm^b M.$ Then the following conditions are equivalent:
                \begin{itemize}
 \item[(i)]  $\reg(G_I(M)) = {e_0 - b +2\choose 2}  - e_1 -1;$
 \item[(ii)] $\ HP_{I,M}(z) = \frac{b+\sum_{i=1}^{e_0-b}z^i}{1-z};$
 \item[(iii)]  $e_1 = {e_0 - b +1 \choose 2};$
 \item[(iv)] $\reg(G_I(M)) =  e_0 -b$ and $G_I(M)$ is a Cohen-Macaulay module.
                \end{itemize}

 Moreover, if one of the above condition holds  then $b = \max\{t\mid IM \subseteq m^tM\}$.
\end{Theorem}
\begin{proof}
The last statement follows from Remark \ref{CY0} (i).

(ii) $\Longleftrightarrow$ (iii) is the Rossi-Valla result (see  Lemma \ref{BDA5} (ii)).

(i) $\Longrightarrow$ (iii): Since  $M$ is a Cohen-Macaulay module, by  Proposition \ref{BDA7},  $\reg(G_I(M)) \leq e_0 - b$. Hence ${e_0 - b +2\choose 2}  - e_1 -1\leq e_0-b$, or equivalently $e_1 \geq {e_0 - b +1 \choose 2}$. By Lemma \ref{BDA5} (i),  we then get  $e_1= {e_0 - b +1 \choose 2}$.

(ii) $\Longrightarrow$ (iv): Note that we always have $e_0 \geq b$. By Remark  \ref{BDA0} (ii), $p(G_I(M)) = e_0 - b - 1$. Applying  Lemma \ref{BDB0}, we get  $e_0-b \leq \reg(G_I(M))$. Combining with   Proposition \ref{BDA7} we can conclude that  $\reg(G_I(M)) =  e_0 -b.$

 By Lemma \ref{BDA6},  $\reg(G_I(M)) = \reg(\overline{G_I(M)})= e_0 - b$.  From (\ref{EB9}) we get $h_{ \overline{G_I(M)}}(0) = b.$ By Lemma \ref{BDA4} (ii) this implies
 $$
h_{ \overline{G_I(M)}}(t) =
\begin{cases}
b \ \ \ \ \ \ \ \ \ \text{if}\ t = 0, \\
t+ b  \ \ \ \ \text{if} \  1\leq t \leq  e_0 -b-1 ,\\
e_0  \ \ \ \ \ \ \ \ \text{if} \ t \geq  e_0 -b.
\end{cases}
$$
Consequently,
\begin{equation*}
\ HP_{\overline{G_I(M)}}(z) = \frac{b+\sum_{i=1}^{e_0-b}z^i}{1-z} = HP_{I,M}(z).
\end{equation*}
Hence $G_I(M) = \overline{G_I(M)}$ and $G_I(M)$ is a Cohen-Macaulay module.

(iv) $\Longrightarrow$ (i): Since $G_I(M)$ is a Cohen-Macaulay module and $\reg(G_I(M) )= e_0 - b$, using Lemma \ref{BDA4} (ii) we get  $HP_{I,M}(z) =  \frac{b+\sum_{i=1}^{e_0-b}z^i}{1-z}.$ By virtue of  the equivalence of (ii) and  (iii), it follows  that $ {e_0 - b +1 \choose 2} = e_1.$ Therefore
$$\begin{array}{ll}
\reg(G_I(M)) =  e -b & = e - b  + {e_0 - b +1 \choose 2} - e_1\\
& =   {e_0 - b +2 \choose 2} - e_1 -1.
\end{array}$$
\end{proof}

The following example shows that  the assumption $G_I(M)$ being a  Cohen-Macalay module in  (iv) of the above theorem is essential.
 \begin{Example}\label{VDB1}{\rm
Let $A = k[[t^3,t^4,t^5]] \cong k[[x,y,z]]/(x^3 - yz,   xz - y^2, x^2y - z^2)$, where $k$ is a field. Let $I = (t^3,t^4)$.
We  have $b = 1$. Using CocoA package \cite{Co}, we can compute
%$$G_I(A) = \frac{A}{(t^3,t^4)}\bigoplus \frac{(t^3,t^4)}{(t^6,t^7,t^8)}\bigoplus  \frac{(t^6,t^7,t^8)}{(t^9,t^{10},t^{11})}\bigoplus\cdots,$$
%hence
$HP_{I,A}(z) = \frac{2+z^2}{1-z}.$
Hence $e_0 = 3$,  $e_1 = 2$. By Lemma \ref{BDA0},  $p(G_I(A)) = 1$. Using also  Lemma \ref{BDB0}, we then get  $\reg(G_I(A)) \geq2$. By Proposition \ref{BDA7},  $\reg(G_I(A)) \leq e_0 -b =2,$ which yields $\reg(G_I(A)) = e_0 -b = 2$.  However  $2= e_1 \neq{ e_0 - b +1 \choose 2}=3. $ Note that, by   Theorem \ref{DLB3},  $ G_I(A)$ in this example cannot be a Cohen-Macalay  ring.}
\end{Example}

\begin{Theorem} \label{DLB4}Let  $M$ be an one-dimensional module and $\depth(M) = 0$. Let $b$ be a positive integer such that $IM \subseteq \mm^bM.$ Then the following conditions are equivalent:
                \begin{itemize}
 \item[(i)]  $\reg(G_I(M)) = {e_0 - b +2\choose 2}  - e_1 -1;$
 \item[(ii)] $HP_{I,M}(z) = \frac{b+\sum_{i=1}^{e_0-b+1}z^i-z^{{e_0 - b +2\choose 2}  - e_1}}{1-z}.$
                \end{itemize}

 Moreover, if one of the above condition holds then $b = \max\{t\mid IM \subseteq m^tM\}$ .
\end{Theorem}
\begin{proof}
(i) $\Longrightarrow$ (ii):
For simplicity we set $\overline{M}: = M/L$,  $\overline{e}_0: = e_0(I,\overline{M})   = e_0$ and $ \overline{e}_1: = e_1(I,\overline{M}),$ where $L: = H^0_\mm(M)$. Analyzing the proof of Theorem \ref{DLA8} we see that the condition (i) implies
\begin{equation}\label{EB12.1}
\reg(G_I(\overline{M})) = e_0 - b,
\end{equation}
\begin{equation} \label{EB13}
 \overline{e}_1 = {e_0 - b +1 \choose 2},
\end{equation}
 and
  \begin{equation} \label{EB10}
\reg(G_I(M)) = \reg(G_I(\overline{M})) + \ell(L).
\end{equation}
By \cite[Proposition 2.3]{RV} and (\ref{EB13}) we get
\begin{equation}
\ell(L) = \overline{e}_1 - e_1 = {e_0 - b +1 \choose 2} -e_1.
\end{equation}
Using Lemma \ref{BDB1}, the equality (\ref{EB10}) implies
\begin{equation} \label{EB11}
HP_K(z) = \sum_{ \reg(G_I(\overline{M})) + 1}^{ \reg(G_I(\overline{M})) + \ell(L)}z^i,
\end{equation}
where $K = \oplus_{n \geq 0}\frac{I^{n+1}M+L\cap I^nM}{I^{n+1}M}.$
Combining  (\ref{EB12.1}),  (\ref{EB13})  and (\ref{EB11})  we get
$$HP_K(z) = \sum_{e_0-b+1}^{{e_0 - b +2\choose 2}  - e_1-1}z^i.$$
 %$,
Using Lemma \ref{BDA5}(ii)  and (\ref{EB12.1}) we have
% ,  it follows that (\ref{EB13}) is equivalent to
$$HP_{I,\overline{M}}(z) = \frac{b+\sum_{i=1}^{e_0-b}z^i}{1-z}.$$
 Hence, using the short exact sequence (\ref{EB14})
we conclude that
\begin{align*}
HP_{I,M}(z) &= HP_{I,\overline{M}}+HP_K(z) \\
& = \frac{b+\sum_{i=1}^{e_0-b}z^i }{1-z} +\sum_{e_0-b+1}^{{e_0 - b +2\choose 2}  - e_1-1}z^i\\
& =  \frac{b+\sum_{i=1}^{e_0-b+1}z^i-z^{{e_0 - b +2\choose 2}  - e_1}}{1-z}.
\end{align*}

(ii) $\Longrightarrow$ (i):  By Lemma \ref{BDA5} (i), ${e_0 - b +2\choose 2}  - e_1 > e_0 - b +1 $. Hence, by  Lemma \ref{BDA0}, $p(G_I(M)) = {e_0 - b +2\choose 2}  - e_1 -1$. By Remark \ref{BDA1} (i),  we obtain
$$ {e_0 - b +2\choose 2}  - e_1 -1 \leq \reg(G_I(M)).$$
Combining with Theorem   \ref{DLA8}  we then get  $\reg(G_I(M)) =  {e_0 - b +2\choose 2}  - e_1 -1.$

For the last statement we see that in this case the equality (\ref{EB13})  holds. By Remark \ref{CY0} (i), $I\overline{M} \nsubseteq \mm^{b+1}\overline{M}$. This implies
$IM \nsubseteq \mm^{b+1}M$ and $b = \max\{t\mid IM \subseteq \mm^tM\}$.
\end{proof}

There are many examples of one-dimension Cohen-Macaulay  rings $(A,\mm)$ and $\mm$-primary ideals  such that  $e_1  ={e_0 - b +1\choose 2}$. Hence, by Theorem \ref{DLB3},  the upper bound ${e_0 - b+2\choose 2}-e_1-1$ is sharp. The following example show that this bound is also attained in the non-Cohen-Macaulay case.

\begin{Example}\label{VDB3}{\rm
Let $A = k[[x,y]]/(x^sy^{u+v},x^{s+1}y^u)$, where $s,u,v \in \NN$ and $v> 0$. Then $G_\mm(A) \cong k[x,y]/(x^sy^{u+v},x^{s+1}y^u)$ and $b = 1$.  It is easy  to see that
$$HP_{\mm,A}(z) = \frac{\sum_{i=0}^{s+u}z^i-z^{s+u+v}}{1-z}, e_0 = s+u,  e_1 = \frac{(s+u)(s+u-1)}{2}-v,$$
  and $\reg(G_\mm(A)) = s + u+v -1$.  These equalities show that all conditions in Theorem \ref{DLB4}
hold.}
\end{Example}

\noindent {\bf ACKNOWLEDGMENT}

The author is grateful to Prof. L. T. Hoa for his guidance and the referee for critical remarks. He also would like to thank Prof. J. Elias for pointing him some errors in the references.

\end{document}